\documentclass[cmp,draft,numbook]{svjour}
\usepackage{amsmath,amssymb,verbatim}

\spnewtheorem*{cor}{Corollary}{\bf}{\it}
\spnewtheorem*{lem}{Lemma}{\bf}{\it}
\spnewtheorem*{pro}{Proposition}{\bf}{\it}
\spnewtheorem*{teo}{Theorem}{\bf}{\it}

\newcommand\al{\alpha}
\newcommand\dd{\partial}
\newcommand\De{\Delta}
\newcommand\de{\delta}
\newcommand\FF{\mathbb F}
\renewcommand\ge{\geqslant}

\newcommand\La{\Lambda}
\newcommand\la{\lambda}
\newcommand\lap{\lambda^{\ts\prime}}
\newcommand\lcd{,\ldots,}
\renewcommand\le{\leqslant}
\newcommand\mup{\mu^{\,\prime}}
\newcommand\ns{\hspace{-.75pt}}

\newcommand\QQ{\mathbb Q}
\renewcommand\phi{\varphi}

\newcommand\Sg{\mathfrak{S}}
\newcommand\si{\sigma}

\newcommand\ts{\hspace{0.75pt}}
\newcommand\xh{\widehat{x}}
\newcommand\ZZ{\mathbb{Z}}

\newcommand\beq{\begin{equation}}
\newcommand\eeq{\end{equation}}
\newcommand\be{\begin{equation*}}
\newcommand\ee{\end{equation*}}

\begin{document}

\title{Macdonald operators at infinity}
\titlerunning{Macdonald operators}

\author{M.\,L.\,Nazarov and E.\,K.\,Sklyanin}
\authorrunning{Nazarov and Sklyanin}

\institute{Department of Mathematics, University of York, 
York YO10 5DD, United Kingdom}

\date{}

\maketitle


\thispagestyle{empty} 

\begin{abstract}
We construct a family of pairwise commuting operators  
such that the Macdonald symmetric functions of infinitely many variables 
$x_1,x_2,\,\ldots$ and of two parameters $q,t$
are their eigenfunctions.
These operators are defined as
limits at $N\to\infty$ of
renormalised Macdonald operators
acting on symmetric polynomials in the variables 
$x_1\lcd x_N$.
They are differential operators
in terms of the power sum variables $p_n=x_1^n+x_2^n+\ldots$ 
and we compute their symbols
by using the Macdonald reproducing kernel.
We express these symbols in terms of
the Hall\ts-Littlewood symmetric functions
of the variables $x_1,x_2,\,\ldots\ $. 
Our result also yields elementary step operators
for the Macdonald symmetric functions. 

\end{abstract}



\newpage


\section*{Introduction}

Over the last two decades
the Macdonald polynomials \cite{M}  were
the subject of 
much attention
in Combinatorics and Representation Theory.
These polynomials are symmetric in the $N$ variables
$x_1\lcd x_N$ and also depend on two parameters
denoted by $q$ and $t\,$. They are labelled by partitions of $0,1,2,\ldots$
with no more than $N$ 
parts. Up to normalization, they can be defined
as eigenfunctions of certain linear operators
acting on the space of all symmetric polynomials
in the variables $x_1\lcd x_N$ with coefficients
from the field $\QQ(q,t)\,$.
These operators have been introduced by Macdonald \cite{M}
as the coefficients of a certain operator 
valued polynomial $D_N(u)$ of degree $N$
in a variable $u$ with the
constant term $1\ts$, see \eqref{dnu}.
In particular, Macdonald has observed that 
all the eigenvalues of the coefficient
of $D_N(u)$ at $u$ are already free from multiplicities. 
Hence this operator coefficient alone can be used to define
the Macdonald polynomials.

It is quite common in Combinatorics to extend various symmetric polynomials
to an infinite countable set of variables. In particular,
the Macdonald polynomials are extended to infinitely many
variables $x_1,x_2,\ldots$ by using the \textit{stability property}
\eqref{macst} of these polynomials and by passing to their limits as 
$N\to\infty\,$. The limits are the Macdonald symmetric functions,
which are labelled by partitions of $0,1,2,\ldots\,\,$.
They have been also studied very well.
In particular, the limit at $N\to\infty$
of a renormalized coefficient of the $D_N(u)$ at $u$
was considered in \cite{M}.
Another expression for the same limit was given in \cite{AMOS},
see also \cite{CJ,GH}.

The limits at $N\to\infty$ of other coefficients of the $D_N(u)$
have so far received less attention. However,
in the remarkable work of Shiraishi \cite{S} the limits of certain 
linear combinations of
all the coefficients were expressed in terms of the 
\textit{vertex operators\/}
associated with an infinite dimensional Heisenberg Lie algebra,
see also \cite{FHHSY,Y}. In a more recent work 
\cite{AK} this result of \cite{S}
has been reformulated by using the well known correspondence \cite{J}
between the vertex operators and the Hall\ts-Littlewood symmetric functions,
which are 
specializations of the Macdonald symmetric functions at $q=0\,$.

In the present article we consider the limits at $N\to\infty$
of linear combinations of all the coefficients
of $D_N(u)\ts$, different from those 
in \cite{AK,FHHSY,S,Y}. Our linear combinations arise naturally
from the theory of the double affine Hecke algebras, see for instance
\cite{SV}. We also express our limits
in terms of the Hall\ts-Littlewood symmetric functions.
Once stated our result can be derived from
those of \cite{AK,S}.
However, we obtained our result independently. 
Moreover, our proof is different and yields 
new identities related to the Hall\ts-Littlewood 
polynomials. Again, the latter can be regarded as 
specializations of the Macdonald polynomials at~$q=0\,$.

The Macdonald polynomials can be regarded as generalizations
of the Jack polynomials which are symmetric in $x_1\lcd x_N$ 
and also depend on a formal parameter $\al\,$.
The latter polynomials are obtained from the former
when $q=t^{\ts\al}$ and $t\to1\,$. 
Then the coefficients of $D_N(u)$
degenerate to the 
Sekiguchi\ts-Debiard differential operators \cite{D,Se}.
In \cite{NS} we studied the limits
of the latter operators at $N\to\infty\,$. 
In the present article we generalize the main result of that work
to the Macdonald case.
However, the methods used here and in \cite{NS} 
are quite~different.

\enlargethispage{6pt}

As an application of our result, we construct 
\emph{elementary step operators\/} for the  
Macdonald symmetric functions.
In terms of the labels, our operators correspond
to decreasing any given non-zero part of a partion by $1$
and to the operation on partitions inverse to that,
see our formulas \eqref{cshift} and \eqref{bshift} respectively.
For~the origins of this construction
see the work \cite{Sk} and references therein. 
For related but different results on the
Macdonald polynomials see the works \cite{KN,KaN} and \cite{LV}.

Here is the plan of the present article. 
In Section \ref{section::symfun}
we recall some basic facts from the theory of symmetric functions,
including the definition of Macdonald polynomials. 
After establishing the basics we state our main result,
which is an explicit expression
for the limit of a renormalized polynomial $D_N(u)$
at $N\to\infty\,$.
Then we explicitly construct
our elementary step operators. 
The proof of our main result 
is given in Section~\ref{section::theproof}.
Our main 
tool is the notion of the symbol of an operator 
relative to the reproducing kernel associated with the
Macdonald polynomials. Using it,
we reduce the proof to a certain determinantal identity
for each $N=1,2,\ldots$ which is
proved in the rest of Section~\ref{section::theproof}.

In this article we generally keep to
the notation of the book \cite{M} 
for symmetric functions. When using results from \cite{M}
we simply indicate their numbers within the book.
For example, the statement (6.9) from Chapter I of the book
will be referred to as [I.6.9] assuming it is from~\cite{M}. 
We do not number our own lemmas, propositions, theorems or corollaries
because here we have only one of each.


\section{Symmetric functions}
\label{section::symfun}


\subsection{Monomial functions and power sums}
\label{subone}

Fix a field $\FF\,$. For any integer $N\ge1$ 
denote by $\La_N$ the $\FF$-algebra of symmetric 
polynomials in $N$ variables $x_1\lcd x_N\,$.
The algebra $\La_N$ is graded by the polynomial degree.
The substitution $x_N=0$ defines
a homomorphism $\La_N\to\La_{N-1}$ preserving the degree.
Here $\La_{\ts0}=\FF\,$.
The inverse limit of the sequence
$$
\La_1\leftarrow\La_2\leftarrow\ldots
\hspace{-8pt}
$$ 
in the category of graded algebras is denoted by $\La\,$. 
The elements of 
$\La$ are called \textit{symmetric functions\/}.    
Following \cite{M} we will introduce some standard bases of $\La\,$.

Let $\la=(\,\la_1,\la_2,\ldots\,)$ be any partition of $\,0,1,2,\ldots\,\,$. 
The number of non-zero parts is called the {\it length\/} of 
$\la$ and is denoted by $\ell(\la)\,$. 
If $\ell(\la)\le N$ then
the sum of all distinct 
monomials obtained by permuting the $N$ variables in
$x_1^{\,\la_1}\ldots x_N^{\,\la_N}$  is denoted by
$m_\la(x_1\lcd x_N)\,$. 
The symmetric polynomials $m_\la(x_1\lcd x_N)$ with
$\ell(\la)\le N$ form a basis of the vector space $\La_N\,$. 
By definition, for $\ell(\la)\le N$
\begin{equation*}
m_\la(x_1\lcd x_N)=
\sum_{1\le i_1<\ldots<i_k\le N}
\ \sum_{\si\in\Sg_k}\ c_\la^{-1}\ 
x_{i_{\si(1)}}^{\,\la_1}\ldots x_{i_{\si(k)}}^{\,\la_k}
\end{equation*}
where we write $k$ instead of $\ell(\la)\,$. Here
$\Sg_k$ is the symmetric group permuting the numbers $1\lcd k$
and
\begin{equation}
\label{cela}
c_\la=k_1!\,k_2!\,\ldots
\end{equation}
if $k_1,k_2,\ldots$ are the respective multiplicites of the parts $1,2,\ldots$
of $\la\,$. Further,
\begin{equation}
\label{monst}
m_\la(x_1\lcd x_{N-1},0)=
\left\{
\begin{array}{cl}
m_\la(x_1\lcd x_{N-1})
&\quad\textrm{if}\quad\,\ell(\la)<N\,;
\\[2pt]
0
&\quad\textrm{if}\quad\,\ell(\la)=N\,.
\end{array}
\right.
\end{equation}
Hence for any fixed partition $\la$ the sequence of polynomials
$m_\la(x_1\lcd x_N)$ with $N\ge\ell(\la)$ 
has a limit in $\La\,$. This limit is called
the \textit{monomial symmetric function\/}
corresponding to $\la\,$. Simply omitting the variables,
we will denote the limit by $m_\la\,$.
With $\la$ ranging over all partitions of $0,1,2\ldots$ 
the symmetric functions $m_\la$ form a basis of the vector space $\La\,$.
Note that if $\ell(\la)=0$ then we set $m_\la=1\,$.

We will be also using another standard basis of the vector space $\La\,$.
For each $n=1,2,\ldots$ denote $p_n(x_1\lcd x_N)=x_1^n+\ldots+x_N^n\,$.
When the index $n$ is fixed
the sequence of symmetric polynomials $p_n(x_1\lcd x_N)$ with
$N=1,2,\ldots$ has a limit in $\La\,$, called
the \textit{power sum symmetric function} of degree $n\,$.
We will denote it by $p_n\,$. 
More generally, for any partition $\la$ put
\begin{equation}
\label{pla}
p_\la=p_{\la_1}\,p_{\la_2}\ldots
\end{equation}
where we set $p_0=1\,$. The elements $p_\la$
form another basis of $\La\,$. In other words,
the elements $p_1,p_2,\ldots$ are
free generators of the commutative algebra $\La$ over $\FF\,$.

In this article we will be using 
the \textit{natural ordering\/} of partitions.
By definition, here $\la\ge\mu$ if
$\la$ and $\mu$ are partitions of the same number and
$$
\la_1\ge\mu_1,\ \,
\la_1+\la_2\ge\mu_1+\mu_2,\ \,
\ldots\ \,.
$$
This is a partial ordering.
Note that by [I.6.9] any monomial symmetric function $m_\mu$
is a linear combination of the symmetric functions 
$p_\la$ where $\la\ge\mu\,$.


\subsection{Hall\ts-Littlewood functions}
\label{subhl}

Choose $\FF$ to be the field $\QQ(t)$ where $t$ is a formal parameter.
Take any partition $\la$ with $\ell(\la)\le N\,$. Using the notation of
\eqref{cela} put
\begin{equation}
\label{vla}
v_\la(t)\,=\,\prod_{i\ge 0}\,\,\prod_{j=1}^{k_i}\,\,
\frac{\,1-t^{\,j}\!}{\,1-t\,\,}
\end{equation}
where $k_{\ts0}=N-\ell(\la)\,$.
Consider the sum of all the $N\ts!$ products obtained from
$$
x_1^{\,\la_1}\ldots x_N^{\,\la_N}
\prod_{1\le i<j\le N}\frac{x_i-t\,x_j}{x_i-x_j}
$$
by permuting $x_1\lcd x_N\ts$. This sum is 
a symmetric polynomial in $x_1\lcd x_N$ with coefficients
from $\ZZ[t]\ts$. Dividing it by $v_\la(t)$ we get
the \textit{Hall\ts-Littlewood symmetric polynomial\/}
$P_\la(x_1\lcd x_N)\,$, see [III.2.1]. All coefficients of
the latter polynomial also belong to $\ZZ[t]$ by [III.1.5].
Furthermore by [III.2.5] similarly to \eqref{monst} we have
$$
P_\la(x_1\lcd x_{N-1},0)=
\left\{
\begin{array}{cl}
P_\la(x_1\lcd x_{N-1})
&\quad\textrm{if}\quad\,\ell(\la)<N\,;
\\[2pt]
0
&\quad\textrm{if}\quad\,\ell(\la)=N\,.
\end{array}
\right.
$$
Hence for any fixed partition $\la$ the sequence of polynomials
$P_\la(x_1\lcd x_N)$ with $N\ge\ell(\la)$ 
has a limit in $\La\,$. This is the
\textit{Hall\ts-Littlewood symmetric function\/} $P_\la\ts$.

Along with the symmetric function $P_\la$ it is convenient to
use the symmetric function $Q_\la$ which is a scalar multiple
of $P_\la\ts$. By definition,
\begin{equation}
\label{qbp}
Q_\la=b_\la(t)\,P_\la 
\end{equation}
where
\begin{equation}
\label{bla}
b_\la(t)\,=\,\prod_{i\ge1}\,\,\prod_{j=1}^{k_i}\,\,
(\ts 1-t^{\,j}\ts)\,.
\end{equation}
We will also use the symmetric polynomial
$$
Q_\la(x_1\lcd x_N)=b_\la(t)\,P_\la(x_1\lcd x_N)\,. 
$$

When we need to distinguish 
$x_1,x_2,\ldots$ from any other variables,
we will write $f(\ts x_1,x_2,\ldots\ts)$ instead of $f\in\La\,$.
Now let $y_1,y_2,\ldots$ be
variables independent~of~the $x_1,x_2,\ldots\,\,$. 
Then by [III.4.4] we have the identity
\begin{equation}
\label{hlkernel}
\prod_{i,j=1}^\infty
\frac{\,1-t\,x_i\ts y_j}{\,1-x_i\ts y_j\,\,\,}
\,=\,
\sum_\la\,Q_\la(\ts x_1,x_2,\ldots\ts)\,P_\la(\ts y_1,y_2,\ldots\ts)
\end{equation}
where $\la$ ranges over all partitions of
$0,1,2\ldots\,$. The product at the left hand side of
this identity is regarded as an infinite sum of
monomials in $x_1,x_2,\ldots$ and in $y_1,y_2,\ldots$ by expanding 
the factor corresponding to $i,j$ as a series at $x_i\ts y_j\to0\,$.

Note that at $\,t=0\,$ both $P_\la(x_1\lcd x_N)$ and $Q_\la(x_1\lcd x_N)$
specialize to the \textit{Schur symmetric polynomial} $s_\la(x_1\lcd x_N)\,$.
Respectively, the symmetric functions
$P_\la$ and $Q_\la$ specialize at $\,t=0\,$
to the \textit{Schur symmetric function} $s_\la\ts$.
The symmetric function $P_\la$ also admits
specialization at $t=1\ts$. By [III.2.4] the latter specialization
coincides with the monomial symmetric function $m_\la\ts$.   

Now take the symmetric function
$Q_\la$ corresponding to the partition $\la=(n)$ with one part only. 
We will denote this symmetric function by $Q_n\ts$. 
By using a variable $u$ independent of $x_1,x_2,\ldots$ and $t$ 
introduce the generating function  
\begin{equation}
\label{qu}
Q(u)\,=\,1+\sum_{n=1}^\infty\,Q_{n}\ts u^n\ts.
\end{equation}
By [III.2.10] then
\begin{equation}
\label{qprod}
Q(u)\,=\,\prod_{i=1}^\infty\,
\frac{\,1-t\,x_i\,u\,}{1-x_i\ts u\,\,\,}\,\ts.
\end{equation}

\medskip\noindent
By taking the logarithm of the infinite product here and then exponentiating,
\begin{equation}
\label{quexp}
Q(u)\,=\,\exp\,\Bigl(\,\,
\sum_{n=1}^\infty\,
\frac{1-t^n\!}{n}\,\ts p_n\ts u^n
\ts\Bigr)\,.
\end{equation}


\subsection{Green polynomials}
\label{subgreen}

The basis of Hall\ts-Littlewood symmetric functions
can be related to the basis of $p_\la$ as follows. Write
\begin{equation}
\label{laxmu}
p_\la\,=\,\sum_\mu\,X_{\la\mu}(t)\ts P_\mu
\end{equation}
where $X_{\la\mu}(t)\in\QQ(t)$ while both $\la$ and $\mu$
are partitions of the same number. By [III.2.7] each $X_{\la\mu}(t)$ 
is a polynomial in the variable $t$ with integral coefficients. Furthermore
by [III.7.7] this polynomial in $t$ is monic and has the degree
$$
n_\mu\,=\,\sum_{i\ge1}\ts(\,i-1\ts)\,\mu_i\,.
$$
The elements 
$$
t^{\,n_\mu}X_{\la\mu}(t^{-1})\in\ZZ[\ts t\ts]
$$
are called the \textit{Green polynomials\/}, 
see [Ex.\ III.7.7] and references therein.

Because $P_\mu$ specializes at $\,t=0\,$ to the
Schur symmetric function $s_\mu\ts$, the value $X_{\la\mu}(0)$
coincides with the value of the irreducible character
of the symmetric group labeled by the partition $\mu$ at the 
conjugacy class labelled by the partition $\la\,$. Moreover, 
there are \textit{orthogonality relations\/} [III.7.3]
\begin{equation}
\label{orthorel}
\sum_\la\,
X_{\la\mu}(t)\ts X_{\la\nu}(t)\ts/z_\la(t)\,=\,\de_{\mu\nu}\,b_\mu(t)
\end{equation}
where
$$
z_\la(t)\,=\,z_\la\,\prod_{i=1}^{\ell(\la)}\,\frac1{1-t^{\ts\la_i}}
$$
while
\begin{equation}
\label{zela}
z_\la=1^{k_1}k_1!\,2^{k_2}k_2!\ts\,\ldots
\end{equation}
in the notation of \eqref{cela}. At $\,t=0\,$ the relations \eqref{orthorel}
specialize to the standard orthogonality relations for the
irreducible characters of symmetric groups.  
Due to \eqref{qbp} and to \eqref{orthorel} the definition \eqref{laxmu} 
of the polynomials $X_{\la\mu}(t)$ implies~that
\begin{equation}
\label{muxla}
Q_\mu\,=\,\sum_\la\,X_{\la\mu}(t)\,p_\la\ts/z_\la(t)\,.
\end{equation}
  

\subsection{Macdonald functions}
\label{submac}

Now let $\FF$ be the field $\QQ(q,t)$
where $q$ and $t$ are formal parameters independent of each other. 
Define a bilinear form $\langle\ ,\,\rangle$ on the vector space $\La$
by setting for any $\la$ and $\mu$
\begin{equation}
\label{macprod}
\langle\,p_\la,p_\mu\ts\rangle=z_\la\,\de_{\la\mu}\,
\prod_{i=1}^{\ell(\la)}\,\frac{1-q^{\ts\la_i}}{1-t^{\ts\la_i}\,}
\end{equation}
in the notation of \eqref{zela}. This form is obviously symmetric
and non-degenerate. By [VI.4.7]
there exists a unique family of elements $M_\la\in\La$ such that
$$
\langle\,M_\la,M_\mu\ts\rangle=0
\quad\text{for}\quad
\la\neq\mu
$$
and such that any $M_\la$ equals $m_\la$ 
plus a linear combination of the elements $m_\mu$
with $\mu<\la$ in the natural partial ordering.
The elements $M_\la\in\La$ are called 
the \textit{Macdonald symmetric functions\/}.
Alternatively, they can be defined as follows.

Take the algebra $\La_N$ of symmetric polynomials in the variables
$x_1\lcd x_N\,$. For each index $i=1\lcd N$ define
the \textit{q\ts-shift operator\/} $T_i$ on the algebra $\La_N$ by
$$
(T_i\ts f)(x_1\lcd x_N)=f(x_1\lcd q\ts x_i\lcd x_N)\,.
$$
Denote by $\De\ts(\ts x_1\lcd x_N)$ the {\it Vandermonde polynomial\/} 
of $N$ variables
$$
\det\Big[x_i^{\,N-j}\Big]{\phantom{\big[}\!\!}_{i,j=1}^N=
\prod_{1\le i<j\le N}(x_i-x_j)\,.
$$
Put
\begin{equation}
\label{dnu}
D_N(u)=
\De\ts(\ts x_1\lcd x_N)^{-1}\cdot
\det\Big[\,
x_i^{\,N-j}\bigl(\ts 1-u\,t^{\,1-j}\,T_i\ts\bigr)
\Big]{\phantom{\big[}\!\!}_{i,j=1}^N
\hspace{-8pt}
\end{equation}
where $u$ is another variable.
The last determinant is defined as the alternated~sum
\begin{equation}
\label{sedeb}
\sum_{\si\in\Sg_N}
(-1)^{\si}\,
\prod_{i=1}^N\,\,\bigl(\,
x_i^{\,N-\si(i)}\bigl(\ts 1-u\,t^{\,1-\si(i)}\,T_i\ts\bigr)\bigr)
\end{equation}
where as usual $(-1)^{\si}$ denotes the sign of permutation $\si\,$.
In every product over $i=1\lcd N$ appearing in \eqref{sedeb}
the operator factors pairwise commute, 
hence their ordering does not matter. By [VI.4.16]
the $D_N(u)$ is a polynomial in 
$u$ with pairwise 
commuting operator coefficients
preserving the space $\La_N\,$.
We will call the restrictions of the
coefficients to the space $\La_N$ the 
\textit{Macdonald operators\/}. 
By [VI.4.15] they
have a common eigenbasis in $\La_N$ 
parametrized by
partitions $\la$ of length $\ell(\la)\le N\ts$. 
These eigenvectors are the 
\textit{Macdonald symmetric polynomials\/}. 

For each $\la$ with $\ell(\la)\le N$
there is an eigenvector denoted by $M_\la(x_1\lcd x_N)$
which is equal to $m_\la(x_1\lcd x_N)$ plus
a linear combination of the polynomials 
$m_\mu(x_1\lcd x_N)$ with $\mu<\la\,$ and $\ell(\mu)\le N\ts$. 
It turns out that each coefficient
in this linear combination does not depend on $N\ts$.
Note that if $\la$ and $\mu$ are any two partitions of the same number
such that $\la\ge\mu\,$, then $\ell(\la)\le\ell(\mu)$ by
[I.1.11]. It follows that the polynomials $M_\la(x_1\lcd x_N)$
have the same \text{stability property} as the 
polynomials $m_\la(x_1\lcd x_N)$ in \eqref{monst}:
\begin{equation}
\label{macst}
M_\la(x_1\lcd x_{N-1},0)=
\left\{
\begin{array}{cl}
M_\la(x_1\lcd x_{N-1})
&\quad\textrm{if}\quad\,\ell(\la)<N\,;
\\[2pt]
0
&\quad\textrm{if}\quad\,\ell(\la)=N\,.
\end{array}
\right.
\end{equation}
In particular, the sequence of polynomials
$M_\la(x_1\lcd x_N)$ with $N\ge\ell(\la)$
has a limit in $\La\,$. This is exactly
the Macdonald symmetric function $M_\la\,$.
Further, the eigenvalues of Macdonald operators acting on $\La_N$ 
are known. By [VI.4.15]  
\begin{equation}
\label{deigen}
D_N(u)\,M_\la(x_1\lcd x_N)\,=\,
\prod_{i=1}^N\,\bigl(\ts 1-u\,q^{\ts\la_i}\ts t^{\,1-i}\ts\bigr)
\cdot M_\la(x_1\lcd x_N)\,.
\end{equation}
Note that $M_\la(x_1\lcd x_N)$ is a homogeneous polynomial of 
degree $\la_1+\la_2+\ldots\,\ts$,
\begin{equation}
\label{qdeg}
T_1\ldots T_N\,M_\la(x_1\lcd x_N)\,=\,
q^{\,\la_1+\la_2+\ldots}\ts\,M_\la(x_1\lcd x_N)\,.
\end{equation}
Hence the operator $T_1\ldots T_N$ on $\La_N$
commutes with every coefficient of $D_N(u)\,$.
Also note that by [VI.4.14] the symmetric function
$M_\la$ admits a specialization at $q=0\,$.
This specialization equals the
Hall\ts-Littlewood symmetric function~$P_\la\,$.  


\subsection{Reproducing kernel\/}
\label{reproker}

In this subsection we will regard the elements of $\La$
as infinite sums of finite products of the variables
$x_1,x_2,\ldots\,\,$. For instance, we have
$$
p_n=x_1^n+x_2^n+\ldots
$$
for any $n\ge1\,$. Like in the identity \eqref{hlkernel},
we will write $f(\ts x_1,x_2,\ldots\ts)$ instead of any $f\in\La$
when we need to distinguish 
$x_1,x_2,\ldots$ from other variables.
Now let $y_1,y_2,\ldots$ be
variables independent of $x_1,x_2,\ldots\,\,$. 
According to [VI.2.7]
with the bilinear form \eqref{macprod} on $\La$
one associates the \textit{reproducing kernel}
\begin{equation}
\label{kernel}
\Pi\,=\prod_{i,j=1}^\infty\, 
\frac{(\ts t\,x_i\ts y_j\ts;q\ts)_\infty}{(\ts x_i\ts y_j\ts;\ts q\ts)_\infty}
\end{equation}
where as usual
\begin{equation}
\label{usual}
(\ts u\ts;q\ts)_\infty=\prod_{k=0}^\infty\,\bigl(\ts1-u\,q^{\ts k}\ts\bigr)\,.
\end{equation}

The property of $\Pi$ most useful for us can be stated as 
the following lemma. For any $f\in\La$
denote by $f^{\ts\ast}$ the operator on $\La$
adjoint to the multiplication by $f$
relative to the bilinear form \eqref{macprod}. 
Note that here $f=f(\ts x_1,x_2,\ldots\ts)\,$.

\begin{lem}
We have
\begin{equation*}
f^{\ts\ast}(\Pi)/\Pi=f(\ts y_1,y_2,\ldots\ts)\,.
\end{equation*}
\end{lem}

\begin{proof}
The commutative algebra $\La$ is 
generated by the elements $p_n$ with $n\ge1\,$.
Hence it suffices to prove the lemma for $f=p_n$ only.
Take the operator $\dd/\dd\ts p_n$ of derivation 
in $\La$ relative to $p_n=p_n(\ts x_1,x_2,\ldots\ts)\,$. 
Then by the definition \eqref{macprod} 
\begin{equation}
\label{pnast}
p_n^{\ts\ast}\,=\,n\,\,\frac{1-q^n}{1-t^n}\,\frac{\dd}{\dd\ts p_n}\,\,.
\end{equation}
On the other hand, by taking the logarithm of \eqref{kernel} 
and then exponentiating,
$$
\Pi\,=\,\exp\,\Bigl(\,\,
\sum_{n=1}^\infty\,\frac1n\,
\frac{1-t^n}{1-q^n}\,\ts
p_n(\ts x_1,x_2,\ldots\ts)\,p_n(\ts y_1,y_2,\ldots\ts)
\ts\Bigr)\ts.
$$
The lemma for $f=p_n$ follows from
the last two displayed equalities.
\qed
\end{proof}


\subsection{Limits of Macdonald operators}

Let $\FF=\QQ(q,t)$ as in the two subsections above.
For every $N\ge1$ let
$\rho_N$ be the homomorphism $\La_N\to\La_{N-1}$
defined by setting $x_N=0\,$, as
in the beginning of Subsection \ref{subone}. Denote
\begin{equation}
\label{cnu}
A_N(u)=
(\ts T_1\ldots T_N)^{\ts-1}\ts{D_N(u)}/{(\ts u\,;t^{-1})_N}
\end{equation}
where as usual
$$
(\ts u\,;q\ts)_N=\prod_{k=0}^{N-1}\,\bigl(\ts1-u\,q^{\ts k}\ts\bigr)\,.
$$
The right hand side of the equation \eqref{cnu} is regarded as
a rational function of $u$ with the values
being the operators acting on the vector space $\La_N\,$.
Due to the stability property \eqref{macst}
of the Macdonald polynomials, 
\eqref{deigen} and \eqref{qdeg} imply 
$$
\rho_N\,A_N(u)=A_{N-1}(u)\,\rho_N
$$
where $A_{\ts0}(u)=1\,$.
So the sequence of $A_N(u)$ with $N\ge1$ has a limit
at $N\to\infty$. This limit can be written as a series
\begin{equation*}
A(u)=1+{A^{\ts(1)}}/{(\ts u\,;t^{-1}\ts)_1}+
{A^{\ts(2)}}/{(\ts u\,;t^{-1}\ts)_2}+\ldots
\end{equation*}
where the leading term equals $1$ by \eqref{qdeg} while
the coefficients $A^{\ts(1)},A^{\ts(2)},\,\ldots$
are certain linear operators acting on 
$\La\,$.
The Macdonald symmetric functions are
joint eigenvectors of these operators. Namely, by \eqref{deigen} we have
the equality
\begin{equation}
\label{aigen}
A(u)\,M_\la\,=\,M_\la\,
\,\prod_{i=1}^\infty\,\,\frac{\,q^{\ts-\la_i}-u\,t^{\,1-i}}{1-u\,t^{\,1-i}}\,\,.
\end{equation}
In particular,
the operators $A^{\ts(1)},A^{\ts(2)},\,\ldots$ 
pairwise commute and are self-adjoint
relative to the bilinear form \eqref{macprod}.
We call them 
the \textit{Macdonald operators at infinity\/}.
Due to \eqref{macst} their definition 
immediately implies that
$$
A^{\ts(k)}M_\la=0
\quad\text{if}\quad
\ell(\la)<k\,.
$$

The operator $A^{\ts(1)}$ has been well studied,
see for instance [VI.4.3]. It follows from 
\eqref{deigen} and \eqref{qdeg}
that for any partition $\la$
$$
A^{\ts(1)}M_\la\,=\,
\sum_{i=1}^\infty\,(\ts q^{\ts-\la_i}-1\ts)\,t^{\,i-1}\cdot M_\la\,.
$$
In particular, all the eigenvalues of the operator
$A^{\ts(1)}$ on $\La$ are pairwise distinct.
By \cite[Eq.\,32]{AMOS} the operator $A^{\ts(1)}$ is equal 
to the coefficient at $1$ of the series in~$u$
$$
\frac1{1-t}\,\exp\,\Bigl(\,\,
\sum_{n=1}^\infty\,
\frac{1-t^n\!}{n}\,\,u^n\ts p_n
\ts\Bigr)\ts
\exp\,\Bigl(\,\,
\sum_{n=1}^\infty\,
(\ts q^{\ts-n}-1\ts)\,u^{\ts-n}\ts\frac{\dd}{\dd\ts p_n}
\ts\Bigr)
\ts-\,\frac1{1-t}\,\ts.
$$
In the next section we will prove
the following general expression for 
every $A^{\ts(k)}$.

\begin{teo}
In the notation \eqref{qbp}
for every\/ $k=1,2,\ldots$ we have
\begin{equation}
\label{basic}
A^{\ts(k)}=\sum_{\ell(\la)=k}
q^{\,-\la_1-\la_2-\ts\ldots}\,Q_\la\,P_\la^{\,\ast}
\end{equation}
where $\la$ ranges over all partitions of length $k\,$.
\end{teo}

By using 
\eqref{qbp},\eqref{muxla} 
the symmetric functions $P_\la$ and $Q_\la$
can be expressed as linear combinations of the functions
$p_\mu$ where both $\la$ and $\mu$ are partitions of the same number.
By substituting into \eqref{basic} and then using \eqref{pla},\eqref{pnast}
one can express every operator $A^{\ts(k)}$
in terms of $p_n$ and $\dd/\dd\ts p_{n}$ where $n=1,2,\ldots\,\,$.

In the case $k=1$ one can also employ the generating function \eqref{qu}. 
In this case by using the equality
\eqref{quexp}, our theorem follows from the 
expression for the operator $A^{\ts(1)}$ given just before stating the theorem.
Furthermore, for any $k\ge1$
one can derive our theorem 
from the results of \cite[Sec.\,3]{AK}
and \cite[Sec.\,9]{S}. In the present article we give 
a proof independent of these results. In particular, our proof
yields new identities for the Hall\ts-Littlewood symmetric polynomials.


\subsection{Step operators\/}
\label{shiftop}

In this subsection we will obtain a corollary to our theorem.
We will also utilize the following particular case of the {\it Pieri rule}
for Macdonald symmetric functions. By [VI.6.24] for any partition $\mu$
the product $p_1\ts M_\mu$ equals the linear combination of
the symmetric functions $M_\la$ with the coefficients
\begin{equation}
\label{bml}
\prod_{j=1}^{i-1}\,
\frac
{\,1-q^{\,\la_j-\la_i}\,t^{\,i-j+1\,}}
{\,1-q^{\,\la_j-\la_i+1}\,t^{\,i-j\,}}
\,\cdot\,
\prod_{j=1}^{i-1}\,
\frac
{\,1-q^{\,\la_j-\la_i+1}\,t^{\,i-j-1\,}}
{\,1-q^{\,\la_j-\la_i}\,t^{\,i-j\,}}
\end{equation}
where $\la$ ranges over all partitions such that 
the sequence $\la_1,\la_2,\ldots$ is obtained from $\mu_1,\mu_2,\ldots$
by increasing one of its terms by $1$ and $i$ is
the index of the term.
 
Further, by [VI.6.19] the above stated equality implies that for any 
partition $\la$ the symmetric function 
$\dd\ts M_\la/\dd\,p_1=p_1^{\ast}\,M_\la\,(1-t)/(1-q)$ 
is equal to the linear combination of the $M_\mu$ with the coefficients
\begin{equation}
\label{clm}
\prod_{j=1}^{\la_i-1}
\frac
{\,1-q^{\,\la_i-j-1}\,t^{\,\lap_j-i+1\,}}
{\,1-q^{\,\la_i-j}\,t^{\,\lap_j-i\,}}
\,\,\cdot
\prod_{j=1}^{\la_i-1}
\frac
{\,1-q^{\,\la_i-j+1}\,t^{\,\lap_j-i\,}}
{\,1-q^{\,\la_i-j}\,t^{\,\lap_j-i+1\,}}
\end{equation}
where $\mu$ ranges over all partitions such that 
the sequence $\mu_1,\mu_2,\ldots$ is obtained from $\la_1,\la_2,\ldots$
by decreasing one of its terms by $1$ and $i$ is
the index of the term. As usual, here $\lap=(\ts\lap_1,\lap_2,\ldots\ts)$ 
is the partition conjugate to $\la\ts$. 

Let us now define the linear operators $B^{\ts(1)},B^{\ts(2)},\,\ldots$
acting on $\La$ by setting
\begin{equation}
\label{svbu}
[\,p_1\ts,A(u)\ts]_{\ts q}
=-\,u\,B(u)\,(1-q)/(1-t)
\end{equation}
where
$$
B(u)={B^{\ts(1)}}/{(\ts u\,;t^{-1}\ts)_1}+
{B^{\ts(2)}}/{(\ts u\,;t^{-1}\ts)_2}+\ldots\,.
$$
At the left hand side of \eqref{svbu} we have
the $q\ts$-\textit{commutator\/} $p_1\ts A(u)-q\,A(u)\,p_1\,$.

Further, define the linear operators $C^{\ts(1)},C^{\ts(2)},\,\ldots$
acting on $\La$ by setting
\begin{equation}
\label{svcu}
[\,A(u)\ts,\dd/\dd\ts p_1\ts]_{\ts q}=-\,u\,C(u)
\end{equation}
where
$$
C(u)={C^{\ts(1)}}/{(\ts u\,;t^{-1}\ts)_1}+
{C^{\ts(2)}}/{(\ts u\,;t^{-1}\ts)_2}+\ldots\,.
$$
Then by the definitions of $B(u)$ and $C(u)$ 
and by \eqref{pnast} we have the relation
\begin{equation}
\label{bucu}
B(u)^{\ts\ast}\ns=C(u)\,.
\end{equation}

Our theorem provides explicit expressions for the operators
$B^{\ts(1)},B^{\ts(2)},\,\ldots$ and 
$C^{\ts(1)},C^{\ts(2)},\,\ldots$
which we state as the following corollary.
The corollary will allow 
to explicitly construct the elementary step 
operators  for Macdonald symmetric functions, see the equalities 
\eqref{bshift}~and~\eqref{cshift} below. 

\begin{cor}
For every $k=0,1,2,\ldots$ we have the equalities
\begin{align}
\label{bk}
B^{\ts(k+1)}&=\,
t^{\ts-k}\sum_{\ell(\mu)=k}\,
q^{\,-\mu_1-\mu_2-\ts\ldots}\,Q_{\mu\,\sqcup\ts1}\ts P_\mu^{\,\ast}\,,
\\[2pt]
\label{ck}
C^{\ts(k+1)}&=\,
t^{\ts-k}\sum_{\ell(\mu)=k}\,
q^{\,-\mu_1-\mu_2-\ts\ldots}\,
P_\mu\,Q_{\mu\,\sqcup\ts1}^{\,\ast}
\end{align}
where $\mu$ ranges over all partitions of length $k$ 
and $\mu\sqcup1$ denotes the partition obtained from $\mu$
by appending one extra part~$1\ts$.
\end{cor}

\begin{proof}
The stated equalities \eqref{bk} and \eqref{ck}
follow from each other due to the relation \eqref{bucu}.
We shall derive the first of the two equalities from our theorem.

Recall that at $q=0$ the Macdonald symmetric function $M_\la$
specializes to the Hall\ts-Littlewood symmetric function $P_\la\ts$.
Hence the expression for $\dd\ts M_\la/\dd\,p_1$ given above
implies 
\begin{equation}
\label{dp}
{\dd\ts P_\la}\ts/{\dd\,p_1}=
\sum_\mu\,\psi_{\la\mu}(t)\,P_\mu
\end{equation}
where $\mu$ ranges over all partitions such that 
the sequence $\mu_1,\mu_2,\ldots$ is obtained from $\la_1,\la_2,\ldots$
by decreasing one of its terms by $1\ts$. Let $i$ is the index of that term.
If $\la_i=1$ then the coefficient \eqref{clm} is $1\ts$. Then
$\psi_{\la\mu}(t)=1$ in particular. If $\la_i>1$ and $q=0$ then
the only factor in the two products over the indices
$j$ in \eqref{clm} comes from the first product
and corresponds to $j=\la_i-1=\mu_i\,$. Then 
$\psi_{\la\mu}(t)=1-t^{\,m}$ where $m=\lap_j-i+1$
is the multiplicity of the part $\mu_i$ in $\mu\,$.
But we will not use any explicit expression for the coefficient
$\psi_{\la\mu}(t)$ with~$\la_i>1\ts$.

By [III.4.9] the equality \eqref{dp} 
established above is equivalent to the equality
\begin{equation}
\label{pq}
(1-t)\ts\,p_1\ts Q_\mu=
\sum_\la\,\psi_{\la\mu}(t)\,Q_\la
\end{equation}
where $\la$ ranges over all partitions such that 
the sequence $\la_1,\la_2,\ldots$ is obtained from $\mu_1,\mu_2,\ldots$
by increasing one of its terms by $1\ts$. 
The latter equality can also be derived from
the multiplication formula \eqref{aom} below, by setting $n=1$~there.

Now for any $k\ge1$ consider
the $q\ts$-commutator $[\,p_1\ts,A(u)\ts]_{\ts q}\,$.
By \eqref{basic}
\begin{align*}
p_1\ts A^{\ts(k)}\ts(1-t)
&=\sum_{\ell(\mu)=k}
q^{\,-\mu_1-\mu_2-\ts\ldots}\,(1-t)\ts\,p_1\,Q_\mu\,P_\mu^{\,\ast}
\\
&=\sum_{\ell(\mu)=k}\sum_\la\,\,
q^{\,-\mu_1-\mu_2-\ts\ldots}\,\psi_{\la\mu}(t)\ts\,Q_\la\,P_\mu^{\,\ast}
\end{align*}
where we use the notation of \eqref{pq}. 
Further, by \eqref{basic} and \eqref{dp}
\begin{align}
\nonumber
q\,\ts [\,A^{\ts(k)},\ts p_1\,]\,(1-t)/(1-q)
&=\sum_{\ell(\la)=k}
q^{\,1-\la_1-\la_2-\ts\ldots}\,Q_\la\,
[\,P_\la^{\,\ast},\ts p_1\ts]\,(1-t)/(1-q)
\\
\nonumber
&=\sum_{\ell(\la)=k}
q^{\,1-\la_1-\la_2-\ts\ldots}\,Q_\la\,
[\ts\,p_1^{\ts\ast}\ts,P_\la\,]^{\ts\ast}\,(1-t)/(1-q)
\\
\nonumber
&=\sum_{\ell(\la)=k}
q^{\,1-\la_1-\la_2-\ts\ldots}\,Q_\la\,
(\,{\dd\ts P_\la}\ts/{\dd\,p_1})^{\ts\ast}
\\
\label{lamusum}
&=\sum_{\ell(\la)=k}\sum_\mu\,\,
q^{\,-\mu_1-\mu_2-\ts\ldots}\,\psi_{\la\mu}(t)\,Q_\la\,P_\mu^{\,\ast}
\end{align}
where the square brackets stand for the usual operator commutator. Hence
\begin{gather*}
[\,p_1\ts,A(u)\ts]_{\ts q}\,(1-t)/(1-q)=
p_1\ts A^{\ts(k)}\ts(1-t)-q\,\ts [\,A^{\ts(k)},\ts p_1\,]\,(1-t)/(1-q)=
\\[6pt]
\sum_{\ell(\mu)=k}
q^{\,-\mu_1-\mu_2-\ts\ldots}\,Q_{\mu\,\sqcup\ts1}\ts P_\mu^{\,\ast}-
\sum_{\ell(\mu)=k-1}
q^{\,-\mu_1-\mu_2-\ts\ldots}\,Q_{\mu\,\sqcup\ts1}\ts P_\mu^{\,\ast}
\end{gather*}
where we use the equality
$\phi_{\mu\,\sqcup\ts1,\ts\mu}(t)=1\ts$. The definition \eqref{svbu}
now implies that 
\begin{align*}
-\,u\,B(u)
=(1-t)\,p_1\,
&+\,\sum_{k=1}^\infty\,
\sum_{\ell(\mu)=k}\,
q^{\,-\mu_1-\mu_2-\ts\ldots}\,Q_{\mu\,\sqcup\ts1}\ts P_\mu^{\,\ast}
\ts/\ts{(\ts u\,;t^{-1}\ts)_k}
\\
&-\ \sum_{k=1}^\infty\,
\sum_{\ell(\mu)=k-1}\,
q^{\,-\mu_1-\mu_2-\ts\ldots}\,Q_{\mu\,\sqcup\ts1}\ts P_\mu^{\,\ast}
\ts/\ts{(\ts u\,;t^{-1}\ts)_k}
\end{align*}

\begin{align*}
\\[-40pt]
&=\ \sum_{k=0}^\infty\,
\sum_{\ell(\mu)=k}\,
q^{\,-\mu_1-\mu_2-\ts\ldots}\,Q_{\mu\,\sqcup\ts1}\ts P_\mu^{\,\ast}
\ts/\ts{(\ts u\,;t^{-1}\ts)_k}
\\
&-\ \sum_{k=0}^\infty\,
\sum_{\ell(\mu)=k}\,
q^{\,-\mu_1-\mu_2-\ts\ldots}\,Q_{\mu\,\sqcup\ts1}\ts P_\mu^{\,\ast}
\ts/\ts{(\ts u\,;t^{-1}\ts)_{k+1}}\,.
\end{align*}
The required equality \eqref{bk} now follows from the relation
$$
(\ts u\,;t^{-1}\ts)_k^{\ts-1}-(\ts u\,;t^{-1}\ts)_{k+1}^{\ts-1}=
-\,u\,t^{\ts-k}\,(\ts u\,;t^{-1}\ts)_{k+1}^{\ts-1}\,.
\eqno{\square}
$$
\end{proof}

Note that in the infinite product over the indices $i$ 
at the right hand side of the equality \eqref{aigen} 
the only factors different from $1$ are those
corresponding to $i=1\lcd\ell(\la)\ts$.
For any such index $i$ consider the product
\begin{equation}
\label{iskip}
\frac{t^{\,1-i}}{1-u\,t^{\,1-i}}\,\ts
\prod_{\substack{j=1\\j\neq\ts i\ts}}^{\ell(\la)}\,
\frac{\,q^{\ts-\la_j}-u\,t^{\,1-j}}{1-u\,t^{\,1-j}}\,\,.
\end{equation}
It follows from \eqref{aigen} and from the definition \eqref{svbu}
that for any given partition~$\mu$
\begin{equation}
\label{bmula}
B(u)\,M_\mu=\sum_{\la}\,B_{\ts\la\mu}(u)\,M_\la
\end{equation}
where $B_{\ts\la\mu}(u)$ equals the product of \eqref{bml} by \eqref{iskip}
and by $1-t\,$,
while $\la$ ranges over all partitions such that 
the sequence $\la_1,\la_2,\ldots$ is obtained from $\mu_1,\mu_2,\ldots$
by increasing one of its terms by $1$ and $i$ is
the index of the term. 

Similarly, \eqref{aigen} and \eqref{svcu} imply that
for any given $\la$
\begin{equation}
\label{clamu}
C(u)\,M_\la=\sum_{\mu}\,C_{\mu\la}(u)\,M_\mu
\end{equation}
where $C_{\mu\la}(u)$ equals the product of \eqref{clm} by \eqref{iskip}
and by $1-q\,$,
while $\mu$ ranges over all partitions such that 
the sequence $\mu_1,\mu_2,\ldots$ is obtained from $\la_1,\la_2,\ldots$
by decreasing one~of its terms by $1$ and $i$ is
the index of the term.
 
Now let the partition $\la$ be fixed. Then for 
the indices $i=1\lcd\ell(\la)$ all the elements 
$q^{\ts-\la_i}\ts t^{\,i-1}$ of the field $\QQ(q,t)$ are pairwise distinct.
Therefore by \eqref{bmula}
for the partition $\mu$ corresponding to any of these indices $i$ we have
\begin{equation}
\label{bshift}
B(\ts q^{\ts-\la_i}\ts t^{\,i-1}\ts)\,M_\mu=
B_{\ts\la\mu}(\ts q^{\ts-\la_i}\ts t^{\,i-1}\ts)\,M_\la
\end{equation}
where the coefficient $B_{\ts\la\mu}(\ts q^{\ts-\la_i}\ts t^{\,i-1}\ts)$
is the product of \eqref{bml} by $1-t$ and by
\begin{equation}
\label{iskipla}
t^{\,1-i}\,
\prod_{j=1}^{\ell(\la)}
\ts\frac1{\ts q^{\,\la_i}-t^{\,i-j}\ts}\,
\prod_{\substack{j=1\\j\neq\ts i\ts}}^{\ell(\la)}\,
(\ts q^{\,\la_i-\la_j}-t^{\,i-j}\ts)\,.
\end{equation}
The left hand side of the equality \eqref{bshift} should be understood
as the value in $\La$ of the rational function $B(u)\,M_\mu$ at the point
$u=q^{\ts-\la_i}\ts t^{\,i-1}\ts$. Similarly, by~\eqref{clamu}
\begin{equation}
\label{cshift}
C(\ts q^{\ts-\la_i}\ts t^{\,i-1}\ts)\,M_\la=
C_{\mu\la}(\ts q^{\ts-\la_i}\ts t^{\,i-1}\ts)\,M_\mu
\end{equation}
where 
$C_{\mu\la}(\ts q^{\ts-\la_i}\ts t^{\,i-1}\ts)$
is the product of \eqref{clm} by $1-q$ and by \eqref{iskipla}.

\newpage

Our definitions \eqref{svbu} and \eqref{svcu}
of the series $B(u)$ and $C(u)$ 
are motivated by the results from \cite[Sec.1]{SV}.
But our definitions employ
the $q\ts$-commutators of the operators 
$p_1$ and $\dd/\dd\ts p_1$
with $A(u)\,$, while in \cite{SV} the usual commutators
have been employed.
Our theorem also provides 
analogues of the equalities \eqref{bk} and \eqref{ck}
for the usual commutators of 
$p_1$ and $\dd/\dd\ts p_1$
with $A(u)\,$.
These analogues however involve summation over the pairs
of partitions $\la$ and $\mu\ts$, see \eqref{lamusum} above.


\section{Proof of the theorem}
\label{section::theproof}


\subsection{Reduction of the proof}

In this subsection we will reduce the proof of our theorem
to a certain determinantal identity
for each $N=1,2,\ldots\,\,$.
By the lemma from Subsection \ref{reproker}
the theorem is equivalent to the equality
\begin{equation*}
\label{syminf}
A(u)(\Pi)/\Pi\,=\,
\sum_\la\,
q^{\,-\la_1-\la_2-\ts\ldots}\,
Q_\la(\ts x_1,x_2,\ldots\ts)\,P_\la(\ts y_1,y_2,\ldots\ts)
/
(\ts u\,;t^{-1}\ts)_{\ell(\la)}
\end{equation*} 
where the coefficients of the series 
$A(u)$ are regarded as operators acting on the
symmetric functions in the variables $x_1,x_2,\ldots\ $.
Here we let the $\la$ range over all partitions of $0,1,2\ldots$
and assume that $(\ts u\,;t^{-1}\ts)_{\ts0}=1\,$.

It suffices to prove for $N=1,2,\ldots$
the restriction of the required functional equality to 
\begin {equation}
\label{xrest}
x_{N+1}=x_{N+2}=\ldots=0\,.
\end{equation}
By the definition of $A(u)$
the restriction of $A(u)(\Pi)/\Pi$
to \eqref{xrest} as of a function 
in $x_1,x_2,\ldots$ equals 
\begin{equation}
\label{pmsp}
A_N(u)(\ts\Pi_N)/\Pi_N
\end{equation}
where we denote
$$
\Pi_N\,=\,
\prod_{i=1}^N\, 
\prod_{j=1}^\infty\,\ts 
\frac{(\ts t\,x_i\ts y_j\ts;q\ts)_\infty}{(\ts x_i\ts y_j\ts;\ts q\ts)_\infty}
\ .
$$ 
By the definition of the 
symmetric function 
$Q_\la(\ts x_1,x_2,\ldots\ts)$ its restriction to \eqref{xrest}
is $Q_\la(\ts x_1\lcd x_N)$ if $\ell(\la)\le N$
and vanishes if $\ell(\la)>N\ts$. 
Hence the restriction of the right hand side of
the required functional equality to \eqref{xrest} is
\begin{equation}
\label{symfininf}
\sum_{\ell(\la)\le N}
q^{\,-\la_1-\la_2-\ts\ldots}\,
Q_\la(\ts x_1\lcd x_N)\,P_\la(\ts y_1,y_2,\ldots\ts)
/
(\ts u\,;t^{-1}\ts)_{\ell(\la)}\,.
\end{equation} 
Due to [VI.2.19]
to prove the equality between \eqref{pmsp} and \eqref{symfininf}
it suffices to set 
$$
y_{N+1}=y_{N+2}=\ldots=0\,.
$$
However we will keep working with the infinite
collection of variables $y_1,y_2,\ldots\ \ts$.
This will simplify the induction argument
in the next subsection. Note that by replacing
in \eqref{symfininf} each variable $x_i$ with $q\ts x_i$ we get 
a sum independent of $q\,$:
\begin{equation}
\label{symq}
\sum_{\ell(\la)\le N}\,
Q_\la(\ts x_1\lcd x_N)\,P_\la(\ts y_1,y_2,\ldots\ts)
/
(\ts u\,;t^{-1}\ts)_{\ell(\la)}\,.
\end{equation} 

Let us compute the function \eqref{pmsp}.
This function depends on the variable $u$ rationally. 
It is also symmetric in 
either of the two collections of variables
$x_1\lcd x_N$ and $y_1,y_2,\ldots\ \ts$. 
It can be obtained
by applying to the identity function $1$
the result of conjugating $A_N(u)$ by the operator of
multiplication by $\Pi_N\ts$, see also \cite[Sec.\,1]{KN}
for a similar argument. By the definition \eqref{dnu} we have
$$
(\ts T_1\ldots T_N)^{\ts-1}\ts{D_N(u)}\,=\,
\De\ts(\ts x_1\lcd x_N)^{-1}\cdot
\det\Big[\,
x_i^{\,N-j}\bigl(\ts T_i^{\,-1}-u\,t^{\,1-j}\ts\bigr)
\Big]{\phantom{\big[}\!\!}_{i,j=1}^N\,.
$$
The last determinant is defined as the alternated~sum
$$
\sum_{\si\in\Sg_N}
(-1)^{\si}\,
\prod_{i=1}^N\,\,\bigl(\,
x_i^{\,N-\si(i)}\bigl(\ts T_i^{\,-1}-u\,t^{\,1-\si(i)}\ts\bigr)\bigr)\,.
$$
Conjugating this sum by $\Pi_N$ amounts to replacing
every $T_i^{\,-1}$ by its conjugate
$$
\prod_{l=1}^\infty\ts
\frac{\,1-q^{\ts-1}\ts t\,x_i\ts y_{\ts l}}
{\,1-q^{\ts-1}\ts x_i\ts y_{\ts l}\ \,}
\ts\cdot\ts T_i^{\,-1}\,,
$$
see the definition \eqref{usual}. Hence we get the sum
$$
\sum_{\si\in\Sg_N}
(-1)^{\si}\,
\prod_{i=1}^N\,\,\Bigl(\,
x_i^{\,N-\si(i)}\ts\Bigl(\,\, 
\prod_{l=1}^\infty\ts
\frac{\,1-q^{\ts-1}\ts t\,x_i\ts y_{\ts l}}
{\,1-q^{\ts-1}\ts x_i\ts y_{\ts l}\ \,}
\ts\cdot\ts T_i^{\,-1}
-u\,t^{\,1-\si(i)}\ts\Bigr)\Bigr)\,.
$$
Here in any single summand
each of the factors corresponding to
$i=1\lcd N$ does not depend on the variables $x_j$ with $j\neq i\,$.
Therefore applying the latter operator sum to 
$1$ amounts to simply deleting each $T_i^{\,-1}$. Thus we get the function
\begin{gather*}
\sum_{\si\in\Sg_N}
(-1)^{\si}\,
\prod_{i=1}^N\,\,\Bigl(\,
x_i^{\,N-\si(i)}\ts\Bigl(\,\, 
\prod_{l=1}^\infty\ts
\frac{\,1-q^{\ts-1}\ts t\,x_i\ts y_{\ts l}}
{\,1-q^{\ts-1}\ts x_i\ts y_{\ts l}\ \,}
-u\,t^{\,1-\si(i)}\ts\Bigr)\Bigr)=
\\[4pt]
\det\Big[\,
x_i^{\,N-j}\ts\Bigl(\, 
\prod_{l=1}^\infty\ts 
\frac{\,1-q^{\ts-1}\ts t\,x_i\ts y_{\ts l}}
{\,1-q^{\ts-1}\ts x_i\ts y_{\ts l}\ \,}
-u\,t^{\,1-j}\ts\Bigr)\ts
\Big]{\phantom{\big[}\!\!}_{i,j=1}^N\,.
\end{gather*}
Dividing by $\De\ts(\ts x_1\lcd x_N)$
and then replacing each variable $x_i$ with $q\ts x_i$~we~get
\begin{equation}
\label{symp}
\De\ts(\ts x_1\lcd x_N)^{-1}\cdot
\det\Big[\,
x_i^{\,N-j}\ts\Bigl(\, 
\prod_{l=1}^\infty\ts 
\frac{\,1-t\,x_i\ts y_{\ts l}}{\,1-x_i\ts y_{\ts l}\ \,}
-u\,t^{\,1-j}\ts\Bigr)\ts
\Big]{\phantom{\big[}\!\!}_{i,j=1}^N\,.
\end{equation}

Thus to prove our theorem it suffices to show
that for $N=1,2,\ldots$
the sum \eqref{symq}
is equal to the ratio \eqref{symp} divided by $(\ts u\,;t^{-1})_N\ts$,
see the definition \eqref{cnu}.
In the next subsection we will reduce the proof of this
equality to a family
of certain identities for symmetric polynomials
in the single collection of variables $x_1\lcd x_N\ts$.
These identities will correspond to 
partitions $\la$ with $0<\ell(\la)<N\,$.


\subsection{Further reduction}

Let us consider the last determinant in \eqref{symp}.
We will be proving by induction on $N$ that this
determinant is equal to the sum \eqref{symq}
multiplied by the Vandermonde polynomial
$\De\ts(\ts x_1\lcd x_N)$ and by $(\ts u\,;t^{-1})_N\ts$.

If $N=0$ then
there is only one term in the sum \eqref{symq}, and this term
is $1\,$. The leading term of the series $A(u)$ is also $1\,$.
Hence we can use the case $N=0$ as the induction base.  
Now take any $N\ge1$ and suppose that
the required equality holds for $N-1$ instead of $N\,$.
For each $i=1\lcd N$ we will for short denote
$$
\De^{(i)}=
\De\ts(\ts x_1\lcd\xh_i\lcd x_N)
$$
where as usual the symbol $\xh_i$ indicates the omitted variable.
Similarly, for any partition $\mu$ with $\ell(\mu)<N$ we will for short denote
$$
Q_\mu^{\ts(i)}=
Q_\mu(x_1\lcd\xh_i\lcd x_N)\ts.
$$

Due to \eqref{qprod} the infinite product over the index 
$l$ in \eqref{symp} equals the sum
$$
1+\sum_{n=1}^\infty\,Q_{n}(\ts y_1,y_2,\ldots\ts)\,x_i^{\ts n}\ts.
$$
Therefore by expanding the last determinant in \eqref{symp}
in its first column and~then employing the induction assumption 
where $u$ and $\la$ are replaced with $u\,t^{-1}$ and $\mu$
respectively, we get the sum
\begin{gather}
\notag
\sum_{i=1}^N\ 
(-1)^{\ts i+1}\,x_i^{\,N-1}\,
\Bigl(\ts1-u+\sum_{n=1}^\infty\,Q_{n}(\ts y_1,y_2,\ldots\ts)\,
x_i^{\ts n}\Bigr)\,
\De^{(i)}
\,\times
\\[4pt]
\label{sumone}
\ \sum_{\ell(\mu)<N}\,
Q_\mu^{\ts(i)}\,P_\mu(\ts y_1,y_2,\ldots\ts)
\!\!\prod_{\ell(\mu)<l<N}\!\!\bigl(\ts1-u\,t^{\ts-l}\ts\bigr)\,. 
\end{gather}

Let us open the brackets in the first of the two lines of the display 
\eqref{sumone} and use the multiplication formula due to Morris \cite{Morris}
\begin{equation}
\label{aom}
Q_{n}(\ts y_1,y_2,\ldots\ts)\,P_\mu(\ts y_1,y_2,\ldots\ts)=
\sum_\la\,\phi_{\la\mu}(t)\,P_\la(\ts y_1,y_2,\ldots\ts)\,,
\end{equation}
see also [III.5.7]. Here $\phi_{\la\mu}(t)\neq0$ only if
\begin{equation}
\label{between}
\la_1\ge\mu_1\ge\la_2\ge\mu_2\ge\ldots
\end{equation}
and
\begin{equation}
\label{lamun}
\la_1-\mu_1+\la_2-\mu_2+\ldots=n\,.
\end{equation}
Then in the notation 
\eqref{cela} the coefficient $\phi_{\la\mu}(t)$ is the product of the 
differences $1-t^{\ts k_i}$ taken over all the
indices $i\ge1$ such that
\begin{equation*}
\label{phiprod}
\lap_i-\mup_i>\lap_{i+1}-\mup_{i+1}
\end{equation*}
where 
$\lap=(\ts\lap_1,\lap_2,\ldots\ts)$ 
and 
$\mup=(\ts\mup_1,\mup_2,\ldots\ts)$
are the conjugate partitions.

In the proof of our theorem we will not
use this explicit expression for the coefficient
$\phi_{\la\mu}(t)\,$. We have reproduced it here for
the sake of completeness. We will use only the fact that
the inequality $\phi_{\la\mu}(t)\neq0$ implies
\eqref{between} and \eqref{lamun}.
We also note that for any fixed partition $\la$ with $\ell(\la)\le N$
and any index $i=1\lcd N$ the multiplication formula \eqref{aom}
implies the decomposition formula
\begin{equation}
\label{mao}
Q_\la(\ts x_1\lcd x_N)\,=\sum_{\ell(\mu)<N}
\phi_{\la\mu}(t)\,x_i^{\ts n}\,Q_\mu^{\ts(i)}
\end{equation}
where the sum is taken over all partitions
$\mu$ with $\ell(\mu)<N$ while
$n$ is determined by the equality \eqref{lamun},
see for instance [III.5.5]~and~[III.5.14].

Using the multiplication formula \eqref{aom}, the sum \eqref{sumone} equals
\begin{gather}
\notag
\sum_{i=1}^N\,\sum_{\ell(\mu)<N}\,
(-1)^{\ts i+1}\,x_i^{\,N-1}\,\De^{(i)}\,Q_\mu^{\ts(i)}
P_\mu(\ts y_1,y_2,\ldots\ts)\,(\ts1-u\ts)
\!\!\!\prod_{\ell(\mu)<l<N}\!\!\!(\ts1-u\,t^{\ts-l}\ts)\ +
\\
\notag
\sum_{i=1}^N\,\sum_{\ell(\mu)<N}\,\sum_{n=1}^\infty
(-1)^{\ts i+1}\,x_i^{\,N-1+n}\,\De^{(i)}\,Q_\mu^{\ts(i)}
\!\!\!\prod_{\ell(\mu)<l<N}\!\!\!(\ts1-u\,t^{\ts-l})\ \times
\\[4pt]
\label{sumtwo}
\sum_\la\,\phi_{\la\mu}(t)\,P_\la(\ts y_1,y_2,\ldots\ts)\,.
\end{gather}
Note that under the condition \eqref{between}
the inequality $n\ge1$ is equivalent to $\la\neq\mu\,$.
Also note that under the condition \eqref{between}
the length $\ell(\la)$ is equal to $\ell(\mu)$ or to $\ell(\mu)+1\,$.
In particular, if $\phi_{\la\mu}(t)\neq0$ in \eqref{sumtwo} 
then $\ell(\la)\le N\ts$. 

Let us now fix any partition $\la$ with $\ell(\la)\le N$ and compare 
the coefficients at $P_\la(\ts y_1,y_2,\ldots\ts)$ in the sum 
displayed in the three lines \eqref{sumtwo}, and in the sum \eqref{symq}
multiplied by $\De\ts(\ts x_1\lcd x_N)$ and $(\ts u\,;t^{-1})_N\ts$.
The latter coefficient always~equals
\begin{equation}
\label{qco}
\De\ts(\ts x_1\lcd x_N)\,Q_\la(\ts x_1\lcd x_N)
\!\!\!\prod_{\ell(\la)\le l<N}\!\!\!(\ts1-u\,t^{\ts-l})\,.
\end{equation}
But when taking the coefficient
in \eqref{sumtwo} we will separately consider three cases.

First suppose that $\ell(\la)=0\,$. In this case there is no
partition $\mu$ satisfying the condition \eqref{lamun} with
$n\ge1\,$. By setting $\ell(\mu)=0$ in the first line of \eqref{sumtwo}
we~get 
$$
\sum_{i=1}^N\,
(-1)^{\ts i+1}\,x_i^{\,N-1}\,\De^{(i)}
\prod_{0\le l<N}(\ts1-u\,t^{\ts-l})
$$
which equals \eqref{qco} with $\ell(\la)=0\,$.
Hence the two coefficients are the same here. 

Next suppose that $\ell(\la)=N\,$. Then
the first line of \eqref{sumtwo} does not contribute
to the coefficient at $P_\la(\ts y_1,y_2,\ldots\ts)$
since $\ell(\mu)<N$ in that line. 
Consider the last two lines of \eqref{sumtwo}.
If $\phi_{\la\mu}(t)\neq0$ there then $\ell(\mu)=N-1$
by the condition $\eqref{between}$, so that
the product over 
$l$ is actually $1\,$. Then the inequality $n\ge1$ holds 
since $\ell(\mu)<\ell(\la)\ts$. Thus
the coefficient at $P_\la(\ts y_1,y_2,\ldots\ts)$
with $\ell(\la)=N$ in \eqref{sumtwo}~equals 
\begin{equation}
\label{nco}
\sum_{i=1}^N\,
\sum_{\ell(\mu)<N}
(-1)^{\ts i+1}\,x_i^{\,N-1+n}\,\De^{(i)}\,Q_\mu^{\ts(i)}\,
\phi_{\la\mu}(t)\,.
\end{equation}
Using the decomposition formula \eqref{mao}, the last displayed sum equals
$$
\De\ts(\ts x_1\lcd x_N)\,Q_\la(\ts x_1\lcd x_N)
$$
and hence coincides with \eqref{qco} in the case $\ell(\la)=N$
under our consideration.

Finally let $0<\ell(\la)<N\ts$. Then the coefficient at 
$P_\la(\ts y_1,y_2,\ldots\ts)$ in \eqref{sumtwo} is
\begin{gather}
\notag
\sum_{i=1}^N\,
(-1)^{\ts i+1}\,x_i^{\,N-1}\,\De^{(i)}\,Q_\la^{\ts(i)}\,(\ts1-u\ts)
\!\!\!\prod_{\ell(\la)<l<N}\!\!\!(\ts1-u\,t^{\ts-l}\ts)\ +
\\
\label{sumthree}
\sum_{i=1}^N\,\sum_{\substack{\ell(\mu)<N\\n\ge1}}
(-1)^{\ts i+1}\,x_i^{\,N-1+n}\,\De^{(i)}\,Q_\mu^{\ts(i)}\,\phi_{\la\mu}(t)
\!\!\!\prod_{\ell(\mu)<l<N}\!\!\!(\ts1-u\,t^{\ts-l})\ .
\end{gather}
The sum displayed in the first of the above two lines 
can be rewritten as
\begin{gather}
\notag
\sum_{i=1}^N\,
(-1)^{\ts i+1}\,x_i^{\,N-1}\,\De^{(i)}\,Q_\la^{\ts(i)}
\!\!\!\prod_{\ell(\la)\le l<N}\!\!\!(\ts1-u\,t^{\ts-l}\ts)\ +
\\
\label{sumfour}
\sum_{i=1}^N\,
(-1)^{\ts i+1}\,x_i^{\,N-1}\,\De^{(i)}\,Q_\la^{\ts(i)}\,
(\ts u\,t^{\,-\ell(\la)}-u\ts)
\!\!\!\prod_{\ell(\la)<l<N}\!\!\!(\ts1-u\,t^{\ts-l}\ts)\ .
\end{gather}
Further, in the second line of the display \eqref{sumthree}
we may have $\phi_{\la\mu}(t)\neq0$ only if 
$\ell(\mu)$ equals $\ell(\la)$ or $\ell(\la)-1\,$.
Therefore the sum in that line can be rewritten as
\begin{gather}
\notag
\sum_{i=1}^N\,\sum_{\substack{\ell(\mu)=\ell(\la)\\\mu\neq\la}}
(-1)^{\ts i+1}\,x_i^{\,N-1+n}\,\De^{(i)}\,Q_\mu^{\ts(i)}\,\phi_{\la\mu}(t)
\!\!\!\prod_{\ell(\la)\le l<N}\!\!\!(\ts1-u\,t^{\ts-l})\ +
\\
\notag
\sum_{i=1}^N\,\sum_{\substack{\ell(\mu)=\ell(\la)\\\mu\neq\la}}
(-1)^{\ts i+1}\,x_i^{\,N-1+n}\,\De^{(i)}\,Q_\mu^{\ts(i)}\,\phi_{\la\mu}(t)
\,u\,t^{\,-\ell(\la)}
\!\!\!\prod_{\ell(\la)<l<N}\!\!\!(\ts1-u\,t^{\ts-l})\ +
\\
\label{sumfive}
\sum_{i=1}^N\,\sum_{\ell(\mu)<\ell(\la)}
(-1)^{\ts i+1}\,x_i^{\,N-1+n}\,\De^{(i)}\,Q_\mu^{\ts(i)}\,\phi_{\la\mu}(t)
\!\!\!\prod_{\ell(\la)\le l<N}\!\!\!(\ts1-u\,t^{\ts-l})\ .
\end{gather}

Using the decomposition formula \eqref{mao},
the sums displayed in the first line of \eqref{sumfour}
and in the first and the third lines of \eqref{sumfive}
add up to the product~\eqref{qco}. 
This product is the coefiicient at 
$P_\la(\ts y_1,y_2,\ldots\ts)$ in the sum \eqref{symq}
multiplied by $\De\ts(\ts x_1\lcd x_N)$ and $(\ts u\,;t^{-1})_N\ts$.
The sums displayed in the second line of
\eqref{sumfour} and in the second line of \eqref{sumfive}
should add up to zero. Let us multiply each of these two sums by
$t^{\,\ell(\la)}$ and divide them by their common factors $u$ and
$$
1-u\,t^{\ts-l}
\quad\ \text{where}\ \quad 
\ell(\la)<l<N\ts.
$$
The proof of our theorem thus reduces to the next combinatorial proposition.

\begin{pro}
For any fixed partition $\la$ with\/ $0<\ell(\la)<N$
we have the identity
\begin{equation*}
\label{theid}
\sum_{i=1}^N\,(-1)^{\ts i+1}\,x_i^{\,N-1}\De^{(i)}\ts
\Bigl(\ts
Q_\la^{\ts(i)}\,\bigl(\ts1-t^{\,\ell(\la)}\ts\bigr)
+
\!\!\!\sum_{\substack{\ell(\mu)=\ell(\la)\\\mu\neq\la}}\!\!\!
x_i^{\ts n}\,Q_\mu^{\ts(i)}\ts\phi_{\la\mu}(t)
\Bigr)=0
\end{equation*}
where $n$ is determined by the partitions $\la$ and $\mu$
via the equality \eqref{lamun}.
\end{pro}

Note that at the left hand side of the above identity 
we have a skew-symmetric polynomial 
in the variables $x_1\lcd x_N$ with the coefficients
from $\ZZ[t]$. Dividing it by the Vandermonde polynomial
$\De\ts(\ts x_1\lcd x_N)$ we get a
symmetric polynomial in $x_1\lcd x_N\ts$.
Our proposition states that the latter
polynomial is actually zero.


\subsection{Finishing the proof}

For any non-negative integer $n$ 
and for any partition $\mu$ with $\ell(\mu)<N$ consider the sum 
\begin{equation}
\label{isum}
\sum_{i=1}^N\,(-1)^{\ts i+1}\,x_i^{\,N-1+n}\De^{(i)}\ts Q_\mu^{\ts(i)}.
\end{equation}
If $n$ is determined by the equality \eqref{lamun}
for any fixed partition $\la$ with $\ell(\la)<N\ts$, 
then the left hand side of the identity in our proposition 
is a linear combination of the sums \eqref{isum} 
with the coefficients $1-t^{\,\ell(\la)}$ or
$\phi_{\la\mu}(t)$ if respectively $\mu=\la$ or $\mu\neq\la$ 
but $\ell(\mu)=\ell(\la)\,$. Note that 
$\phi_{\la\la}(t)=1\,$. 
In particular, $\phi_{\la\la}(t)\neq1-t^{\,\ell(\la)}\,$.

Denote by $F_{\mu,n}(x_1\lcd x_N)$
the alternated sum of $N\ts!$ products obtained from
\begin{equation}
\label{altpro}
x_1^{\,\mu_1}\ldots x_{N-1}^{\,\mu_{N-1}}\,x_N^{\,N-1+n}
\prod_{1\le i<j<N}(\,x_i-t\,x_j\ts)
\end{equation}
by permuting $x_1\lcd x_N\ts$.
Here we use the signs of permutations for alternation.
This sum is a skew-symmetric polynomial in $x_1\lcd x_N$ with
coefficients from $\ZZ[t]\ts$. 
By performing the summation
first over the permutations which map $x_N$ to $x_i$
and then over the indices $i=1\lcd N$ one shows that the product
\begin{equation}
\label{defmun}
(-1)^{\ts N+1}\ts
F_{\mu,n}(x_1\lcd x_N)\cdot 
b_\mu(t)/v_\mu(t)
\end{equation}
equals the sum \eqref{isum}.
Here one uses only the definition of the polynomial
$Q_\mu^{\ts(i)}$, see the beginning of Subsection \ref{subhl}.

Below we shall prove that if
$\mu$ satisfies the conditions $\ell(\mu)=\ell(\la)$ and \eqref{between}
while $n$ is determined by \eqref{lamun} then 
$F_{\mu,n}(x_1\lcd x_N)$ is a linear
combination of the products
\begin{equation}
\label{depnu}
\De\ts(\ts x_1\lcd x_N)\,P_\nu(\ts x_1\lcd x_N)
\end{equation}
where $\nu$ is a partition, $\ell(\nu)\le N$ 
and $\nu<\la$ in the natural 
ordering. Hence the arguments from the previous two subsections imply
that the difference between the left and right hand sides
of the required equality \eqref{basic} is a linear combination
of the operators $P_\nu\,P_\la^{\,\ast}$
where $\nu<\la\,$. However,
this difference must be self-adjoint relative to the 
bilinear form \eqref{macprod} on the vector space $\La\,$.
Indeed, the left hand side of \eqref{basic} is
self-adjoint by definition, 
while the right hand side is self-adjoint due to \eqref{qbp}.
Therefore the difference equals zero.

Now fix any 
$\mu$ satisfying the conditions $\ell(\mu)=\ell(\la)$ and
\eqref{between}. Determine the integer 
$n$ by \eqref{lamun}. Take any monomial in
the variables $x_1\lcd x_N$ resulting from opening
the brackets in \eqref{altpro}. 
Due to the alternation it suffices to consider only those monomials
where $x_1\lcd x_N$ occur with distinct degrees.
By rearranging these degrees in the descending order
we get a monomial
\begin{equation}
\label{numon}
x_1^{\,\nu_1+N-1}\ldots\, x_{N-1}^{\,\nu_{N-1}+1}\,x_N^{\,\nu_N}
\end{equation}
where $\nu_1\ge\ldots\ge\nu_N\ge0\,$. For any $k=1\lcd N-1$ the sum
of the first $k$ degrees
$$
(\ts\nu_1+N-1\ts)+\ldots+(\ts\nu_k+N-k\ts)
$$
does not exceed the maximum of the following two sums\ts:
\begin{equation}
\label{firstsum}
(\ts\mu_1+N-2\ts)+\ldots+(\ts\mu_k+N-k-1\ts)
\end{equation}
and
\begin{equation}
\label{secondsum}
(\ts\mu_1+N-2\ts)+\ldots+(\ts\mu_{k-1}+N-k\ts)+(\ts N-1+n\ts)\,;
\end{equation}
see the proof of the property [III.2.6] of 
Hall\ts-Littlewood polynomials
for a similar argument.
Hence if \eqref{firstsum} is the maximum of the two sums then 
due to \eqref{between} 
$$
\nu_1+\ldots+\nu_k
\,\le\,
\mu_1+\ldots+\mu_k-k
\,<\,
\la_1+\ldots+\la_k\,.
$$
If \eqref{secondsum} is the maximum 
then due to \eqref{between} and \eqref{lamun} 
\begin{gather*}
\nu_1+\ldots+\nu_k
\,\le\,
\mu_1+\ldots+\mu_{k-1}+n\,=
\\[4pt]
(\ts\la_1+\ldots+\la_k\ts)
-(\ts\mu_k-\la_{k+1}\ts)-\ldots-(\ts\mu_{N-1}-\la_N\ts)
\,\le\,
\la_1+\ldots+\la_k\,.
\end{gather*}
For any $k\le\ell(\la)$ the last inequality is strict
because $\ell(\mu)=\ell(\la)\,$. Thus $\nu<\la\,$.

By the definition of the Schur symmetric polynomial
corresponding to the partition $\nu=(\nu_1\lcd\nu_N,0,0,\ldots)$
the alternated sum of $N\ts!$ products 
obtained by permuting $x_1\lcd x_N$ in the monomial \eqref{numon}
is equal to the product
\begin{equation}
\label{desnu}
\De\ts(\ts x_1\lcd x_N)\,s_\nu(\ts x_1\lcd x_N)\,.
\end{equation}
Hence we have now proved that 
$F_{\mu,n}(x_1\lcd x_N)$ is a linear
combination of the products \eqref{desnu}
where $\nu<\la\,$. In the latter statement,
the products \eqref{desnu} can be replaced with the
respective products \eqref{depnu} by using the property [III.2.6].

Thus we have completed the proof of our theorem.
Further, by the definitions \eqref{vla} and \eqref{bla} the
factor $b_\mu(t)/v_\mu(t)$ appearing in the product \eqref{defmun} equals
$$
(\ts 1-t\ts)^{\ts N}\!
\prod_{j=1}^{N-\ell(\mu)\!}\ts\!(\ts1-t^{\,j}\ts)^{-1}\,.
$$
In particular, this factor is the same for all products
\eqref{defmun} such that $\ell(\mu)=\ell(\la)\,$.
Dividing the identity in our proposition 
by this factor and by $(-1)^{\ts N+1}$ we get
$$
\bigl(\ts1-t^{\,\ell(\la)}\ts\bigr)\,F_{\la,0}(x_1\lcd x_N)\ +
\!\!\!\sum_{\substack{\ell(\mu)=\ell(\la)\\\mu\neq\la}}\!\!\!
\phi_{\la\mu}(t)\,F_{\mu,n}(x_1\lcd x_N)=0
$$
for any fixed partition $\la$ such that $0<\ell(\la)<N\ts$.
Here the positive integer $n$ is determined by the partitions $\la$ and $\mu$
via the equality \eqref{lamun}.

It would be interesting to prove this identity
without using any properties of Macdonald operators, for instance by employing
the methods of \cite[Chapter 7]{L}.
It would be also interesting to find a link between this identity
and that from~\cite{DJ}.


\section*{\normalsize\bf Acknowledgements}

We are grateful to M.\,V.\,Feigin, G.\,I.\,Olshanski, J.\,Shiraishi 
and E.\,Vasserot for helpful comments on this work.
The first and the second named of us have been supported
by the EPSRC grants EP/I\,014071 and EP/H\ts000054 respectively.



\end{document}